\documentclass{amsart}
\usepackage{amsthm}

\newtheorem{theorem}{Theorem}
\newtheorem{lemma}[theorem]{Lemma}
\newtheorem{corollary}[theorem]{Corollary}
\newtheorem{remark}{Remark}
\newtheorem{example}{Example}

\title[Almost everywhere convergence of subsequences\ldots]{Almost everywhere convergence of subsequences of Walsh-N\"orlund means}
\author[I. Blahota]{Istv\'an Blahota}
\address{I. Blahota, Institute of Mathematics and Computer Sciences, University of Ny\'\i regyh\'aza, H-4400 Ny\'\i regyh\'aza, S\'ost\'oi street 31/b, Hungary}
\email{blahota.istvan@nye.hu}

\begin{document}
\begin{abstract}
Let $\{q_{k}: k\in\mathbb{N}\}$ be a non-increasing and convex sequence of non-negative numbers such that $q_{0}>0$. Our main result is to prove almost everywhere convergence 
\[
	t_{a_{n}}(f)\to f
\]	
of $f\in L_{1}(G)$ as $n\to\infty$, using several $\{a_{n}: n\in\mathbb{P}\}$ subsequences of positive integers for N\"orlund means (defined by sequence $q$) on the Walsh-Paley system with weaker conditions than it was known before.
\end{abstract}

\maketitle

\section{Definitions and notations}

Let $\mathbb{P}$ be the set of positive natural numbers and $\mathbb{N}:=\mathbb{P}\cup \{0\}$.
Let denote the discrete cyclic group of order
$2$ by ${\mathbb Z}_2$. The group operation is the modulo $2$ addition. Let every subset
be open. The normalized Haar measure $\mu$ on ${\mathbb Z}_{2}$ is given
in the way that $\mu (\{ 0\})=\mu (\{ 1\})=1/2$. That is, the measure of a singleton is $1/2$. $G :=
\overset{\infty}{\underset{k=0}{{\times}}} {\mathbb Z}_{2},$
$G$
is called the dyadic group. The elements of dyadic group $G$ are the $0,1$ sequences. That is,
$x=(x_0,x_1,\dots,x_k,\dots)$ with $x_k\in \{0,1\}$,  where $k\in \mathbb{N}$.

The group operation on $G$ is the coordinate-wise addition (denoted by $+$),
the normalized Haar measure  $\mu $ is the product measure and the topology is the product topology. For an other topology on the dyadic group see e.g. \cite{BSU}.

Dyadic intervals are defined in the usual way
\[
I_0(x):=G,\quad I_n(x):=\{y\in G:y=(x_0,\dots,x_{n-1},y_n,y_{n+1},\dots)\}
\]
for $x\in G$ and $n\in \mathbb{P}$. They form a base for the neighbourhoods of $G$.

For representing elements of $G$ in the interval $[0, 1)$, we use Fine's map is defined by
\[
|x|:=\sum_{i=0}^\infty \frac{x_i}{2^{i+1}} \quad \textrm{for all }x\in G.
\]

Let $L_{p}(G)$ denote the usual Lebesgue spaces on $G$, where $1\leq p<\infty$  (with the corresponding norm $\Vert .\Vert_p$).

Now, we introduce some concepts of Walsh-Fourier analysis. The Rademacher functions are defined on $x\in G$ and $k\in\mathbb{N}$ as
\[
r_k(x):=(-1)^{x_k}.
\]
The sequence of the Walsh-Paley functions is the product system of the Rademacher functions. Namely,
every $n\in\mathbb{N}$ can uniquely be expressed in the number system based $2$, in the form
\[
n=\sum_{k=0}^\infty n_{k}2^{k},\quad  n_{k}\in \{ 0,1\} \quad (k\in\mathbb{N}),
\]
where only a finite number of $n_{k}$'s differ from zero.
Let the order of $n\in\mathbb{N}$ be denoted by $\vert n\vert :=\max \{ k\in
\mathbb{N}: n_{k}\neq 0\}.$  It means $2^{|n|}\leq n< 2^{|n|+1}$ and $|n|=[\log_{2}(n)]$. The Walsh-Paley functions are $w_0:=1$ and 
\[
w_n(x):=\prod_{k=0}^{\infty }r^{n_k}_k(x)
=(-1)^{\sum_{k=0}^{\vert n\vert }n_kx_k}
\]
for $n\in\mathbb{P}$.

The $k$th Fourier-coefficient, the $n$th partial sum of the Fourier series, the $n$th Dirichlet kernel, the $n$th Fej\'er kernel and the $n$th Fej\'er mean are defined as
\begin{eqnarray*}
	\hat f(k):=
	\int\limits_{G} f w_{k} d\mu,
	\quad
	S_{n}(f):=\sum_{k=0}^{n-1}\hat f(k)w_{k}
	\quad D_{n}:=
	\sum_{k=0}^{n-1}w_{k}, \quad D_0:=0,
\end{eqnarray*}
\begin{eqnarray*}
	K_{n}:=\frac{1}{n}\sum_{j=1}^{n}D_{j},\quad
	\sigma_{n}(f):=\frac{1}{n}\sum_{j=1}^{n}S_{j}(f),
\end{eqnarray*}
where $k\in\mathbb{N}$ and $n\in\mathbb{P}$.

Let $\{ q_{k}: k\in\mathbb{N}\}$ be a sequence of non-negative numbers. The $n$th N\"orlund means (as  generalizations of the Fej\'er mean) of the Walsh(-Paley)-Fourier series is defined by
\[
	t_{n}(f):=\frac{1}{Q_{n}}\sum_{k=1}^{n}q_{n-k}S_k(f),
\]
where $Q_n:=\sum_{k=0}^{n-1}q_{k}$ and $n\in\mathbb{P}$. It is always assumed that $q_0>0$ and 
\[
\lim_{n\to\infty}Q_{n}=\infty.
\]
Let us define the $n$th N\"orlund kernel as
\begin{eqnarray*}
	F_{n}:=\frac{1}{Q_{n}}\sum_{j=1}^{n}q_{n-j}D_{j}.
\end{eqnarray*}

It is easily seen that
\[
	\sigma_{n}(f;x)=\int_{G}f(u+x)K_{n}(u)d\mu(u)
\]
and
\[
	t_{n}(f;x)=\int_{G}f(u+x)F_{n}(u)d\mu(u)
\]
where $x,u\in G$. 

We say that sequence $\{q_{k}:k\in\mathbb{N}\}$ is convex if 
\[
	2q_{k+1}\leq q_{k}+q_{k+2}\quad (\Leftrightarrow q_{k+1}-q_{k+2}\leq q_{k}-q_{k+1})
\] 
holds for all $k\in\mathbb{N}$.
\begin{remark}
	From now on throughout the article $c$ denotes a positive absolute constant, which may vary at different appearances.
\end{remark}

\section{N\"orlund means on the Walsh-Fourier system}

A number of scientists have dealt with N\"orlund means with respect to the one dimensional Walsh-Paley system. Including but not limited to M\'oricz and Siddiqi \cite{MS}, Fridli, Manchanda and Siddiqi \cite{FMS}, Baramidze, Persson, Tangrand and Tephnadze \cite{BPTT}, Areshidze and Tephnadze \cite{AT} and Goginava and K. Nagy \cite{GN}. Additional articles have been written on multidimensional Walsh-Paley systems, as well as Vilenkin systems (in one and multidimensional) which are more general than the Walsh-Paley system, but the discussion of these is beyond the scope of our article.

The results in several mentioned articles have in common that the defining sequences $\{ q_{k}: k\in\mathbb{N}\}$ are monotonic. And it is common that in these papers we have assumptions. In non-increasing cases we usually assume that $1/Q_{n}=O(1/n)$. Our main result is Theorem \ref{main}. What we suppose in this convergence theorem is a weaker condition for several subsequences of N\"orlund means. In Section \ref{examples} we give an application of the Theorem \ref{main} to a special subsequence, which clearly show the strengths of the Theorem in applications.

In the following, we present a result that is directly related to this article.

Originally, the next theorem (see also Corollary \ref{BNPT_corollary}) was proved for bounded Vilenkin systems.
\begin{theorem}[Baramidze, Nadirashvili, Persson and Tephnadze \cite{BNPT}]\label{BNPT_theorem}
	Let $f\in L_{1}(G)$.  Let $\{q_{k}: k\in\mathbb{N}\}$ be a non-increasing sequence of non-negative numbers, where $q_{0}>0$. Let us suppose that  
	\begin{equation*}
		cn\leq Q_{n}.
	\end{equation*}
	Then
	\[
	t_{n}(f)\to f
	\]
	as $n\to\infty$, almost everywhere.
\end{theorem}

\section{Auxiliary results}	

\begin{lemma}[See e.g. \cite{SWSP}]\label{egyperx}
	\[
		|D_{n}(x)|\leq \min\left\{n,\frac{2}{x}\right\}
	\]
\end{lemma}
\begin{lemma}[Yano \cite{Y1}]
	Let $n\in \mathbb{P}$, then
	\[
	\left\|K_{n}\right\|_{1}\leq 2.
	\]\end{lemma}
In 2018, Toledo improved Yano's classical result.
\begin{lemma}[Toledo \cite{Tol}]\label{Tol}  
	\[
	\sup_{n\in\mathbb{P}}\|K_{n}\|_{1}= \frac{17}{15}.
	\]
\end{lemma}
\begin{lemma}[G\'at \cite{G1}]\label{lgat}
	\[
	\int_{\bar{I}_{j}}\sup_{n\geq2^{j}}|K_{n}(x)|d\mu(x)\leq c
	\]
\end{lemma}

The following kernel decomposition result was originally proved in \cite{B1} for so-called matrix transform kernels, which are generalizations of N\"orlund kernels.
\begin{lemma}[Blahota \cite{B1}]\label{lemma2}
	Let $n\in\mathbb{P}$ and $\{q_{k}: k\in\mathbb{N}\}$ be a sequence of real numbers. Then
	\begin{align*}	
		F_{n}=&D_{2^{|n|}}-w_{2^{|n|}-1}K_{2^{|n|}-1}\left(2^{|n|}-1\right)q_{n-1}/Q_{n}\\
		&+w_{2^{|n|}-1}\sum_{k=1}^{2^{|n|}-2}K_{k}k\left( q_{n-2^{|n|}+k+1}-q_{n-2^{|n|}+k}\right)/Q_{n}\\
		&+r_{|n|}\sum_{k=1}^{n-2^{|n|}}D_{k}q_{n-2^{|n|}-k}/Q_{n}.
	\end{align*}
\end{lemma}

\section{Estimates of N\"orlund kernels}

The next corollary follows directly from Lemma \ref{lemma2}.
\begin{corollary}\label{cor}
	Let $\{a_{n}: n\in\mathbb{P}\}$ be a sequence of positive integers. Let $\{q_{k}:k\in\mathbb{N}\}$ be a sequence of non-negative numbers. Then
	\begin{align*}	
		\left|F_{a_{n}}\right|\leq&\left|D_{2^{|a_{n}|}}\right|+\left|K_{2^{|a_{n}|}-1}\right|\left(2^{|a_{n}|}-1\right)q_{a_{n}-1}/Q_{a_{n}}\\
		&+\sum_{k=1}^{2^{|a_{n}|}-2}|K_{k}|k\left|q_{a_{n}-2^{|a_{n}|}+k+1}-q_{a_{n}-2^{|a_{n}|}+k}\right|/Q_{a_{n}}\\
		&+\sum_{k=1}^{a_{n}-2^{|a_{n}|}}|D_{k}|q_{a_{n}-2^{|a_{n}|}-k}/Q_{a_{n}}.
	\end{align*}
\end{corollary} 
\begin{lemma}\label{main_lemma}
	Let $\{q_{k}: k\in\mathbb{N}\}$ be a non-increasing and convex sequence of non-negative numbers, where $q_{0}>0$. Let $\{a_{n}: n\in\mathbb{P}\}$ be a sequence of positive integers, for which 
	\begin{equation}\label{cond}
		c\left(a_{n}-2^{|a_{n}|}\right)\log(a_{n})\leq     Q_{a_{n}}.
	\end{equation}
 Then
	\[
		\int_{\bar{I}_{j}}\sup_{a_{n}\geq2^{j}}\left|F_{a_{n}}(x)\right|d\mu(x)\leq c.
	\]
\end{lemma}
\begin{proof}
	From Corollary \ref{cor} we obtain
	\begin{align*}
		\int_{\bar{I}_{j}}\sup_{a_{n}\geq2^{j}}&\left|F_{a_{n}}(x)\right|d\mu(x)\leq\\
		&\int_{\bar{I}_{j}}\sup_{a_{n}\geq2^{j}}\left|D_{2^{|a_{n}|}}(x)\right|d\mu(x)\\
		&+\int_{\bar{I}_{j}}\sup_{a_{n}\geq2^{j}}\left(\left|K_{2^{|a_{n}|}-1}(x)\right|\left(2^{|a_{n}|}-1\right)q_{a_{n}-1}/Q_{a_{n}}\right)d\mu(x)\\
		&+\int_{\bar{I}_{j}}\sup_{a_{n}\geq2^{j}}\left(\sum_{k=1}^{2^{|a_{n}|}-2}|K_{k}(x)|k\left(q_{a_{n}-2^{|a_{n}|}+k}-q_{a_{n}-2^{|a_{n}|}+k+1}\right)/Q_{a_{n}}\right)d\mu(x)\\
		&+\int_{\bar{I}_{j}}\sup_{a_{n}\geq2^{j}}\left(\sum_{k=1}^{a_{n}-2^{|a_{n}|}}|D_{k}(x)|q_{a_{n}-2^{|a_{n}|}-k}/Q_{a_{n}}\right)d\mu(x)\\
		=:&F_{1}+F_{2}+F_{3}+F_{4}.
	\end{align*}		

	Since $D_{2^{|a_{n}|}}(x)=0$ if $x\in \bar{I}_{|a_{n}|}\supseteq\bar{I}_{j}$, so $F_{1}=0$.
	
	From the non-increasing monotonicity we obtain
	\begin{align}\label{first}
		\left(2^{|a_{n}|}-1\right)q_{a_{n}-1}\leq\sum_{k=0}^{2^{|a_{n}|}-2}q_{k}\leq\sum_{k=0}^{a_{n}-1}q_{k}=Q_{a_{n}},
	\end{align}	
	the Abel transformation implies
	\begin{align}\label{second}
		\sum_{k=1}^{2^{|a_{n}|}-2}k\left(q_{a_{n}-2^{|a_{n}|}+k}-q_{a_{n}-2^{|a_{n}|}+k+1}\right)=&\sum_{k=1}^{2^{|a_{n}|}-1}q_{a_{n}-k}-\left(2^{|a_{n}|}-1\right)q_{a_{n}-1}\nonumber\\
		\leq&\sum_{k=a_{n}-2^{|a_{n}|}+1}^{a_{n}-1}q_{k}\leq Q_{a_{n}}
	\end{align}
	and
	\begin{align}\label{third}
	\sum_{k=1}^{2^{j}-1}k\left(q_{k}-q_{k+1}\right)=&\sum_{k=1}^{2^{j}-1}q_{k}-\left(2^{j}-1\right)q_{2^{j}}\nonumber\\	
	\leq&\sum_{k=0}^{2^{j}-1}q_{k}= Q_{2^{j}}.
\end{align}	

	Let $b_{n}:=2^{|a_{n}|}-1$. Then $a_{n}<2^{|a_{n}|+1}=2(b_{n}+1)$. So  if $a_{n}\geq 2^{j}$, then $b_{n}>2^{j-1}-1$, that is $b_{n}\geq2^{j-1}$. But because of the definition of $b_{n}$, it means that $b_{n}\geq 2^{j}-1$. Then Inequality \eqref{first}, Lemma \ref{Tol} and Lemma \ref{lgat} imply
	\begin{align*}
		F_{2}&\leq		\int_{\bar{I}_{j}}\sup_{a_{n}\geq2^{j}}\left(q_{a_{n}-1}\left(2^{|a_{n}|}-1\right)/Q_{a_{n}}\right)\sup_{a_{n}\geq2^{j}}\left|K_{2^{|a_{n}|}-1}(x)\right|d\mu(x)\\
		&\leq \int_{\bar{I}_{j}}\sup_{a_{n}\geq2^{j}}\left|K_{2^{|a_{n}|}-1}(x)\right|d\mu(x)\\		&\leq \int_{\bar{I}_{j}}\sup_{b_{n}\geq2^{j}-1}\left|K_{b_{n}}(x)\right|d\mu(x)\\		&\leq \int_{\bar{I}_{j}}\sup_{b_{n}\geq2^{j}}\left|K_{b_{n}}(x)\right|d\mu(x)+\int_{\bar{I}_{j}}\left|K_{2^{j}-1}(x)\right|d\mu(x)\leq c.
	\end{align*}
	
	From convexity, using inequality $q_{k+1}-q_{k+2}\leq q_{k}-q_{k+1}$ exactly $a_{n}-2^{|a_{n}|}$-times, we get 		
	\begin{align}\label{fourth}	
		q_{a_{n}-2^{|a_{n}|}+k}-q_{a_{n}-2^{|a_{n}|}+k+1}\leq q_{k}-q_{k+1}.
	\end{align}	
	Inequalities \eqref{second}, \eqref{third}, \eqref{fourth}, Lemma \ref{Tol} and Lemma \ref{lgat} yield
	\begin{align*}
		F_{3}\leq&\int_{\bar{I}_{j}}\sup_{a_{n}\geq2^{j}}\left(\sum_{k=1}^{2^{j}-1}k|K_{k}(x)|\left(q_{a_{n}-2^{|a_{n}|}+k}-q_{a_{n}-2^{|a_{n}|}+k+1}\right)/Q_{a_{n}}\right)d\mu(x)\\	&+\int_{\bar{I}_{j}}\sup_{a_{n}\geq2^{j}}\left(\sum_{k=2^{j}}^{2^{|a_{n}|}-2}k|K_{k}(x)|\left(q_{a_{n}-2^{|a_{n}|}+k}-q_{a_{n}-2^{|a_{n}|}+k+1}\right)/Q_{a_{n}}\right)d\mu(x)\\
		\leq&\int_{G}\sum_{k=1}^{2^{j}-1}k|K_{k}(x)|\left(q_{k}-q_{k+1}\right)/Q_{2^{j}}d\mu(x)\\		&+\int_{\bar{I}_{j}}\sup_{i\geq2^{j}}|K_{i}(x)|d\mu(x)\sup_{a_{n}\geq2^{j}}\sum_{k=2^{j}}^{2^{|a_{n}|}-2}k\left(q_{a_{n}-2^{|a_{n}|}+k}-q_{a_{n}-2^{|a_{n}|}+k+1}\right)/Q_{a_{n}}\\
		\leq&\sup_{k\in\{1,\dots,2^{j}-1\}}\|K_{k}\|_{1}+\int_{\bar{I}_{j}}\sup_{i\geq2^{j}}|K_{i}(x)|d\mu(x)\leq c.
	\end{align*}
	
	Let us observe $F_{4}$. From the monotonicity of the sequence $q$ we get
	\begin{align*}
		\sum_{k=1}^{a_{n}-2^{|a_{n}|}}|D_{k}(x)|q_{a_{n}-2^{|a_{n}|}-k}/Q_{a_{n}}\leq q_{0}\sum_{k=1}^{a_{n}-2^{|a_{n}|}}|D_{k}(x)|/Q_{a_{n}},
	\end{align*}
	so, using Lemma \ref{egyperx} and Condition \eqref{cond} we have
	\begin{align*}
		F_{4}&\leq q_{0} \int_{\bar{I}_{j}}\sup_{a_{n}\geq2^{j}}\left(\sum_{k=1}^{a_{n}-2^{|a_{n}|}}|D_{k}(x)|/Q_{a_{n}}\right)d\mu(x)\\
		&\leq q_{0} \sup_{a_{n}\geq2^{j}}\left(\left(a_{n}-2^{|a_{n}|}\right)/Q_{a_{n}}\right)\int_{2^{-j}}^{1}\frac{2}{x}dx\\
		&= c \sup_{a_{n}\geq2^{j}}\left(\left(a_{n}-2^{|a_{n}|}\right)/Q_{a_{n}}\right)j\\	&\leq c \sup_{n\in\mathbb{P}}\left(\left(a_{n}-2^{|a_{n}|}\right)\log 
		(a_{n})/Q_{a_{n}}\right) \leq  c.
	\end{align*}
\end{proof}
\begin{lemma}\label{norm}
	Let $\{q_{k}: k\in\mathbb{N}\}$ be a non-increasing sequence of non-negative numbers, where $q_{0}>0$. Let $\{a_{n}: n\in\mathbb{P}\}$ be a sequence of positive integers, for which 
	\begin{equation}\label{cond_}
		c\left(a_{n}-2^{|a_{n}|}\right)\log(a_{n})\leq     Q_{a_{n}}.
	\end{equation}
	Then
	\[
		\|F_{a_{n}}\|_{1}\leq c.
	\]
\end{lemma}
\begin{proof}
	We will use assumptions of Lemma \ref{main_lemma}, except the convexity of sequence $\{q_{k}: k\in\mathbb{N}\}$.
	
	From Corollary \ref{cor} and Lemma \ref{Tol} we get inequalities
	\begin{align*}	
		\|F_{a_{n}}\|_{1}\leq&\|D_{2^{|a_{n}|}}\|_{1}+\|K_{2^{|a_{n}|}-1}\|_{1}\left(2^{|a_{n}|}-1\right)q_{a_{n}-1}/Q_{a_{n}}\\
		&+\sum_{k=1}^{2^{|a_{n}|}-2}\|K_{k}\|_{1}k\left|q_{a_{n}-2^{|a_{n}|}+k+1}-q_{a_{n}-2^{|a_{n}|}+k}\right|/Q_{a_{n}}\\
		&+\sum_{k=1}^{a_{n}-2^{|a_{n}|}}\|D_{k}\|_{1}q_{a_{n}-2^{|a_{n}|}-k}/Q_{a_{n}}\\
		\leq&1+17/15\left(2^{|a_{n}|}-1\right)q_{a_{n}-1}/Q_{a_{n}}\\
		&+17/15\sum_{k=1}^{2^{|a_{n}|}-2}k\left|q_{a_{n}-2^{|a_{n}|}+k+1}-q_{a_{n}-2^{|a_{n}|}+k}\right|/Q_{a_{n}}\\
		&+\sum_{k=1}^{a_{n}-2^{|a_{n}|}}\|D_{k}\|_{1}q_{a_{n}-2^{|a_{n}|}-k}/Q_{a_{n}}.
	\end{align*}
	Using inequalities \eqref{first}, \eqref{second} and \eqref{cond_} we have
	\begin{align*}	
		\|F_{a_{n}}\|_{1}&\leq c+\sum_{k=1}^{a_{n}-2^{|a_{n}|}}\|D_{k}\|_{1}q_{a_{n}-2^{|a_{n}|}-k}/Q_{a_{n}}\\
		&\leq c+q_{0}/Q_{a_{n}}\sum_{k=1}^{a_{n}-2^{|a_{n}|}}\log_{2}(k+1)\\
		&\leq c+c(a_{n}-2^{|a_{n}|})\log (a_{n}-2^{|a_{n}|}+1)/Q_{a_{n}}\\
		&\leq c.
	\end{align*}
\end{proof}

\begin{lemma}\label{nyolc}
	With the conditions of Lemma \ref{main_lemma} the operator $T(f):=\sup_{n\in\mathbb{P}}\left|t_{a_{n}}(f)\right|$ is of weak type $(L_1,L_1)$ and of strong type $(L_p,L_p)$ for each $1<p\le\infty$.
\end{lemma}
\begin{proof}
	The strong type $(L_{\infty}, L_{\infty})$ property of the operator $T$ is a trivial consequence of Lemma \ref{norm}. Since it is $\sigma$-sublinear, then by a standard argument the fact that it is quasi-local (Lemma \ref{main_lemma}) gives that it is also of weak type $(L_{1}, L_{1})$. Finally, the interpolation theorem of Marcinkiewicz for sublinear operators completes the proof.
\end{proof}

\section{Almost everywhere convergence}

\begin{theorem}\label{main}
	Let $f\in L_{1}(G)$. With the conditions of Lemma \ref{main_lemma}
	\[
		t_{a_{n}}(f)\to f
	\]
	as $n\to\infty$, almost everywhere.
\end{theorem}
\begin{proof}
	From Lemma \ref{main_lemma} we get that 
	operator $T$ is quasi-local. (See e.g. \cite{G1}).
	
	Let $f^{*}:=\sup_{n\in\mathbb{N}}|S_{2^{n}}(f)|$ be the
	maximal function of the integrable function $f\in L_{1}(G)$.
	Then $H(G):=\{f\in L_{1}(G):f^{*}\in L_{1}(G)\}$.
	By standard argument (see e.g. \cite{SWSP}) and by the quasi-locality one can prove that the operator $T$ is of type $(H, L_{1})$ which means that $\|T(f)\|_{1}\leq c \|f\|_{H}$ for	all $f\in H(G)$. (See \cite{F} and \cite{Sch}). Also by standard argument (see e.g. \cite{SWSP}) and by the help of the quasi-locality and Lemma \ref{nyolc} we have that for all $f\in L_{1}(G)$ the almost everywhere convergence holds. 
\end{proof}
\begin{remark}
	If $\{q_{k}: k\in\mathbb{N}\}$ is a non-increasing sequence of non-negative numbers, where $q_{0}>0$ and $Q_{n}\ge cn$ then the statement of Theorem \ref{BNPT_theorem} holds. But if $Q_{n}\leq cn$, then 
	\begin{align*}	
		\sup_{n\in\mathbb{P}}&\left(\left(n-2^{|n|}\right)\log(n)/Q_{n}\right)\\
		\geq&c\lim_{n\to\infty} \left(2^{n}-1-2^{[\log_{2}(2^{n}-1)]}\right)\log\left(2^{n}-1\right)/(2^{n}-1)\\
		=&c\lim_{n\to\infty}\left(2^{n}-1-2^{n-1}\right)\log\left(2^{n}-1\right)/(2^{n}-1)\\
		=&c\lim_{n\to\infty}\frac{(2^{n-1}-1)n}{2^{n}-1}=\infty.
	\end{align*} 
	That is why we use expression
	\[
	\left(n-2^{|n|}\right)\log(n)/Q_{n}	
	\]
	only for subsequences (see Condition \eqref{cond}). For some subsequences we can state the same statement with weaker condition (see Theorem \ref{main}, Corollary \ref{BNPT_corollary} and examples from Section \ref{examples}).
\end{remark}
\begin{corollary}[Baramidze, Nadirashvili, Persson and Tephnadze \cite{BNPT}]\label{BNPT_corollary}
	Let $f\in L_{1}(G)$.  Let $\{q_{k}: k\in\mathbb{N}\}$ be a non-increasing sequence of non-negative numbers, where $q_{0}>0$. 
	Then
	\[
	t_{2^{n}}(f)\to f
	\]
	as $n\to\infty$, almost everywhere.
\end{corollary}
\begin{proof}
	Apply Theorem \ref{main} for $a_{n}:=2^{n}$. Then $a_{n}-2^{|a_{n}|}=0$, so Condition \eqref{cond} is fulfilled. During the proof of Lemma \ref{main_lemma}, convexity was used only to certify Inequality \eqref{fourth}. But in this situation even the equality $q_{a_{n}-2^{|a_{n}|}+k}-q_{a_{n}-2^{|a_{n}|}+k+1}=q_{k}-q_{k+1}$ applies trivially (without assuming convexity).
\end{proof}

\section{Example}\label{examples}

\begin{example}
	Suppose that $f\in L_{1}(G)$. Let  $\{q_{n}: n\in\mathbb{N}\}$  non-increasing and convex sequences of non-negative numbers, where $q_{0}>0$.
	If 
	\[
	cn^{\beta+1}\leq Q_{2^{n}+[n^{\beta}]}
	\]
	holds, then 
	\[
	t_{2^{n}+[n^{\beta}]}(f)\to f
	\]	
	as $n\to\infty$ almost everywhere, where $\beta$ real number is fixed.
\end{example}
\begin{proof}
	Let $n_{0}<n$, where $n_{0}$ satisfies the	inequality $n_{0}^{\beta}<2^{n_{0}}$. Let $a_{n}:=2^{n}+[n^{\beta}]$. Then
	\[
	n=[\log_{2}2^{n}]\leq[\log_{2}a_{n}]<\log_{2}(2^{n}+2^{n})=n+1,
	\]
	so $|a_{n}|=[\log_{2}a_{n}]=n$.	Therefore we have  	
	\begin{align*}
		\left(a_{n}-2^{|a_{n}|}\right)&\log\left(a_{n}\right)/Q_{a_{n}}\\
		&\leq\left(2^{n}+[n^{\beta}]-2^{n}\right)\log\left(2^{n}+[n^{\beta}]\right)c/n^{\beta+1}\\
		&=cn[n^{\beta}]/n^{\beta+1}\leq c.
	\end{align*}
	It means that in this situation conditions of Theorem  \ref{main} are fulfilled.
\end{proof}
\begin{remark}
	Condition $cn^{\beta+1}\leq Q_{a_{n}}$ is not as strict as $c 2^{n}\leq Q_{a_{n}}$ which would come from Theorem \ref{BNPT_theorem}.
\end{remark}

\section{Open question}

Is the statement analogous to Theorem \ref{main} true for matrix transform means (which are generalizations of the N\"orlund means)? Likewise, it is definitely not appropriate, because the estimate of $F_3$ does not work since (without further conditions) we can not leave the supremum out of the first term of the sum.


\end{document}